\documentclass[11pt,reqno]{amsproc}
\usepackage[margin=1in]{geometry}

\usepackage{amsmath, amsthm, amssymb}
\usepackage[colorlinks=true, pdfstartview=FitV, linkcolor=blue,citecolor=blue, urlcolor=blue]{hyperref}
\usepackage[abbrev,lite,nobysame]{amsrefs}
\usepackage{times}
\usepackage[usenames,dvipsnames]{color}
\usepackage{dsfont}

\def\eps{\varepsilon}
\def\e{{\rm e}}
\def\dist{{\rm dist}}

\def\dd{{\rm d}}
\def\ddt{{\frac{\dd}{\dd t}}}

\def\R {\mathbb{R}}
\def\u {\boldsymbol{u}}

\def\N {\mathbb{N}}

\def\AA {{\mathcal A}}

\def\RR {{\mathcal R}}

\def\TT {{\mathbb T}^2}

\DeclareMathOperator*{\esup}{ess\,sup}

\def\de{{\partial}}

\newtheorem{proposition}{Proposition}[section]
\newtheorem{theorem}[proposition]{Theorem}
\newtheorem{corollary}[proposition]{Corollary}
\newtheorem{lemma}[proposition]{Lemma}
\theoremstyle{definition}
\newtheorem{definition}[proposition]{Definition}
\newtheorem{remark}[proposition]{Remark}

\numberwithin{equation}{section}

\title[Uniformly attracting sets for critical SQG]{Uniformly attracting limit sets for the critically dissipative SQG equation}

\author[P.~Constantin]{Peter Constantin}
\address{Department of Mathematics, Princeton University, Princeton, NJ 08544, USA}
\email{const@math.princeton.edu}

\author[M.~Coti Zelati]{Michele Coti Zelati}
\address{Department of Mathematics, University of Maryland, College Park, MD 20742, USA}
\email{micotize@umd.edu}

\author[V.~Vicol]{Vlad Vicol}
\address{Department of Mathematics, Princeton University, Princeton, NJ 08544, USA}
\email{vvicol@math.princeton.edu}

\subjclass[2000]{35Q35}

\keywords{Surface quasi-geostrophic equation, nonlinear maximum principle, De Giorgi, global attractor.}

\begin{document}

\begin{abstract}
We consider the global attractor of the critical SQG semigroup $S(t)$ on 
the scale-invariant space $H^1(\TT)$. It was shown in~\cite{CTV13} that 
this attractor is finite dimensional, and that it attracts uniformly bounded 
sets in $H^{1+\delta}(\TT)$ for any $\delta>0$, leaving open the question 
of uniform attraction in $H^1(\TT)$. In this paper we prove the uniform attraction 
in $H^1(\TT)$, by combining ideas from DeGiorgi iteration and nonlinear maximum principles.
\end{abstract}


\maketitle


\section{Introduction}
We consider the critical surface quasi-geostrophic (SQG) equation 
\begin{equation}\tag{SQG} \label{eq:SQG}
\begin{cases}
\de_t\theta +\u \cdot \nabla \theta+\kappa\Lambda\theta=f\\
\u = \RR^\perp \theta = \nabla^\perp \Lambda^{-1} \theta\\
\theta(0)=\theta_0
\end{cases}
\end{equation}
posed on $\TT=[0,1]^2$, where $\kappa\in (0,1]$ measures the strength of the dissipation,  $\theta_0(x)$ is the initial datum, and $f(x)$ is a 
time-independent force. As usual, $\nabla^\perp= (-\partial_2,\partial_1)$ and $\Lambda = (-\Delta)^{1/2}$ is the Zygmund operator. We assume that the datum and the force have zero mean, i.e. 
$
\int_{\TT} f(x) \dd x = \int_{\TT} \theta_0(x) \dd x=0,
$ 
which immediately implies that \[ \int_{\TT} \theta(x,t) \dd x = 0\] for all $t>0$. 

In this manuscript we consider the dynamical system $S(t)$ generated by \eqref{eq:SQG} 
on the scale-invariant space $H^1(\TT)$. The main result of this paper 
establishes the existence of a compact global attractor for $S(t)$, which uniformly attracts bounded 
set in $H^1(\TT)$. 
\begin{theorem}\label{thm:globalattra}
Let $f\in L^\infty(\TT)\cap H^1(\TT)$. The dynamical system $S(t)$ generated by \eqref{eq:SQG} 
on $H^1(\TT)$ possesses a unique global attractor $\AA$ with the
following properties:
\begin{enumerate}
	\item $S(t)\AA=\AA$ for every $t\geq0$, namely $\AA$ is invariant.
	\item $\AA\subset H^{3/2}(\TT)$, and is thus compact.
	\item For every bounded set $B\subset H^1(\TT)$,
	$$
	\lim_{t\to\infty}\dist(S(t)B,\AA)=0,
	$$
	where $\dist$ stands for the usual Hausdorff semi-distance between sets given by the
	$H^1(\TT)$ norm.
	\item $\AA$ is minimal in the class of $H^1(\TT)$-closed attracting set and 
	maximal in the class of $H^1(\TT)$-bounded invariant sets.
	\item $\AA$ has finite fractal dimension.
\end{enumerate}
\end{theorem}

The study of the long time behavior of hydrodynamics models in terms of finite dimensional global attractors
is closely related to questions regarding the turbulent behavior of viscous fluids, especially
in terms of statistical properties of solutions and invariant measures (see, e.g.~\cites{CF85,CFT85,CFT88,CF88,Hale,FJK96,FMRT01,FHT02,SellYou,T3} and references therein). Several fluids equations have been treated from this point of view, in contexts including 2D turbulence \cites{CF85,CFT85,CFT88,CF88,FHT02,FMRT01,IMT04,JT92,Z97,T3}, the 3D Navier-Stokes equations and regularizations thereof~\cites{Sell96,GT97,CTV07,CZG15 ,CF06,KT09}, and several other
geophysical models~\cites{CT03,FPTZ12,JT15,CD14,CTV13,WT13}, just to mention a few.

Concerning the critical SQG equation, in~\cite{CTV13} the authors have obtained the existence of a finite dimensional invariant attractor 
\[\widetilde \AA \subset H^{3/2}(\TT)\]
such that all point orbits converge to this attractor
$$
\lim_{t\to \infty}\dist(S(t)\theta_0, \widetilde \AA)=0, \qquad \forall \theta_0\in H^1(\TT),
$$
and all bounded sets in a slightly smoother space contract onto it
\begin{align}\label{eq:cptattr}
\lim_{t\to \infty}\dist(S(t)B, \widetilde \AA)=0, \qquad \forall B\subset H^{1+\delta}(\TT),\quad \delta>0.
\end{align}
The question of uniform attraction in $H^1(\TT)$ remained open in~\cite{CTV13}, and is now answered in the positive by 
Theorem~\ref{thm:globalattra}. Moreover, it is in fact not hard to verify that
\begin{align}\label{eq:A:tilde:A}
\AA=\widetilde\AA.
\end{align}
Indeed, since $\AA$ attracts bounded subsets
of $H^1(\TT)$ and $\widetilde\AA$ is invariant, we have
$$
\dist(\widetilde\AA,\AA)=\dist (S(t)\widetilde\AA,\AA)\to 0, \qquad \text{as } t\to \infty,
$$
implying $\widetilde\AA\subset \AA$, since $\AA$ is closed. On the other hand, by the invariance 
of $\AA\subset H^{3/2}(\TT)$ and \eqref{eq:cptattr}, we have
$$
\dist(\AA,\widetilde\AA)=\dist (S(t)\AA,\widetilde\AA)\to 0, \qquad \text{as } t\to \infty,
$$
proving the reverse inclusion  $\AA\subset \widetilde\AA$. Henceforth, equality holds in \eqref{eq:A:tilde:A}.

Comparing Theorem~\ref{thm:globalattra} to the results in~\cite{CTV13}, the new ingredient of this manuscript is to obtain an absorbing ball for the dynamics on $H^1(\TT)$. That is, we prove the existence of a ball 
$
B_a \subset H^1(\TT)
$
such that for any bounded set $B \subset H^1(\TT)$, there exists $t_B\geq 0$ with 
\[
S(t) B \subset B_a
\]
for all $t\geq t_B$. The first difficulty here is that the space $H^1$ is critical, i.e. $\|\cdot \|_{\dot{H}^1}$ is invariant under the natural scaling of \eqref{eq:SQG}, and thus the time of local existence of a solution arising from an initial datum $\theta_0 \in H^1(\TT)$ is not known to depend merely on $\|\theta_0\|_{H^1}$ (rather, it may depend on the rate of decay of the Fourier coefficients, such as the rate at which $|k| |\hat \theta_0(k)| \to 0$ as $|k| \to \infty$). The second difficulty comes from the fact that the Sobolev embedding of $H^1(\TT)$ into $L^\infty(\TT)$ fails, and thus we may not directly consider the evolution of the $L^\infty$ norm of the solution. 

To overcome these difficulties we proceed in three steps:
\begin{enumerate}
\item First, we use the $L^2$ to $L^\infty$ regularization given by the DeGiorgi iteration~\cites{CV10a,CD14,Sil10a} to obtain an $L^\infty$ absorbing set (cf.~Theorem~\ref{thm:absLinf}), with entry time that depends only on $\|\theta_0\|_{L^2}$ and on $\|f\|_{L^2\cap L^\infty}$  (cf.~Theorem~\ref{thm:Linfest}). 
\item Second, we use a quantitative $L^\infty$ to $C^\alpha$ regularization~\cite{CZV14} via nonlinear maximum principles~\cite{CV12} to obtain a $C^\alpha$ absorbing set (cf.~Theorem \ref{thm:absLinf}) with entry time that depends only on $\|\theta_0\|_{L^\infty}$  (the solution already lies inside the $L^\infty$ absorbing set) and on $\|f\|_{L^\infty}$ (cf.~Theorem~\ref{thm:Calphaest}). 
\item Lastly, we use~\cite{CTV13} to show the existence of an $H^1$ absorbing set (cf.~Theorem~\ref{thm:H1abs}) with entry time that depends only on $\|\theta_0\|_{C^\alpha}$ (the solution already lies in the $C^\alpha$ absorbing set) and on $\|f\|_{L^\infty \cap H^1}$.
\end{enumerate}
The existence of the global attractor then follows from the $H^{3/2}$ absorbing ball estimate obtained in~\cite{CTV13}. The remainder of the properties (i)--(v) stated in Theorem~\ref{thm:globalattra} follow along the lines of \cites{CFT85,Hale,PZ,SellYou,T3}, as summarized in Section~\ref{sec:globattr} below.

Lastly, we note that recently in~\cite{CD14} the authors have shown that the dynamics of weak $L^2(\TT)$ solutions to \eqref{eq:SQG} possesses a strong global attractor $\AA_{L^2}$, which is a compact subset in $L^2(\TT)$. The proof in~\cite{CD14} uses the DeGiorgi regularization ideas of~\cite{CV10a}, the weak continuity property of the nonlinearity in \eqref{eq:SQG} for $L^\infty$ weak solutions (which may be established along the lines of~\cites{CCFS,CTV12}), and the compactness argument of~\cite{C09}. As noted in~\cite{CD14}, we have that $\AA \subset \AA_{L^2}$, but it is not clear whether the two attractors coincide, which remains an interesting open problem.


\section{The dynamical system generated by SQG}

We recall the following well-posedness result which summarizes the local in time existence and regularization results of~\cites{CCW00,CC04,Dong10,Ju07,Miu06,Wu07} and the global in time regularity established in~\cites{CV10a,CTV13,CV12,FPV09,KN09,KNV07}:
\begin{proposition}\label{prop:WP}
Assume that {$f\in L^\infty\cap H^1$}. Then, for all initial data $\theta_0\in H^1$ the initial value problem \eqref{eq:SQG}
admits a unique  global solution
\[
\theta\in C([0,\infty);H^1)\cap L^2_{loc}(0,\infty;H^{3/2}).
\]
Moreover, $\theta$ satisfies the energy inequality
\begin{equation}\label{eq:energyin}
\| \theta(t)\|^2_{L^2}+\kappa\int_0^t \|\Lambda^{1/2}\theta(s)\|_{L^2}^2\dd s\leq \|\theta_0\|^2_{L^2}
+\frac{1}{c_0\kappa}\|f\|^2_{L^2}t, \qquad \forall t\geq 0.
\end{equation}
and the decay estimate
\begin{equation}\label{eq:expdecayL2}
\|\theta(t)\|_{L^2}\leq \|\theta_0\|_{L^2}\e^{-c_0\kappa t}+\frac{1}{c_0\kappa}\|f\|_{L^2}, \qquad \forall t\geq 0,
\end{equation}
where $c_0>0$ is a universal constant. If furthermore $\theta_0\in L^\infty$, then cf.~\cites{CC04,CTV13} we have
\begin{equation}\label{eq:expdecayLinf}
\|\theta(t)\|_{L^\infty}\leq \|\theta_0\|_{L^\infty}\e^{-c_0\kappa t}+\frac{1}{c_0\kappa}\|f\|_{L^\infty} , \qquad \forall t\geq 0.
\end{equation}
\end{proposition}

Proposition \ref{prop:WP} translates into the existence of
the solution operators
\[
S(t):H^1\to H^1
\]
acting as
\[
\theta_0\mapsto S(t)\theta_0=\theta(t), \qquad  \forall t\geq0.
\]
Since the forcing term $f$ is time independent, the family
$S(t)$ fulfills the semigroup property
\[
S(t+\tau)=S(t)S(\tau), \qquad \forall t,\tau\geq0.
\]
However, no continuous dependence estimate in $H^1$ is available, since the existence of solutions has been obtained as a stability result of the equation posed in $H^{1+\delta}$ (see \cites{Ju07,Miu06} for details). Consequently, it is not clear whether $S(t):H^1\to H^1$ is continuous in the $H^1$ topology for each fixed $t>0$. Along the lines of the classical references \cites{CFT85,Hale, SellYou, T3}, the theory of infinite-dimensional
dynamical systems has been adapted to more general classes of operators in recent
years (see \cites{C09,CCP12, CZ13, PZ}). It turns out that continuity for fixed $t>0$ is only needed to prove
invariance of suitable attracting sets, while their existence holds under no continuity assumptions 
on $S(t)$. We shall however see in  Section \ref{sec:globattr} that invariance of the attractor may be nonetheless recovered. We recall that: 
\begin{definition}
A set $B_{a} \subset H^1$ is said to be an absorbing set for the semigroup $S(t)$ on $H^1$ if for every bounded set $B\subset H^1$ 
there exists an entering time $t_B\geq 0$ (depending only $B$) such that
\[
S(t)B\subset B_{a}
\]
for all $t\geq t_B$.
\end{definition}
The absorbing set, besides giving a first rough estimate of the dissipativity of the system, is the crucial
preliminary step needed to prove the existence of the global attractor. In particular, a \emph{sufficient}
condition for the existence of the global attractor is the existence of a
\emph{compact} absorbing set.

\section{De Giorgi iteration yields an \texorpdfstring{$L^\infty$}{L infinity} absorbing set}

The first step towards the proof of the existence of a regular uniformly absorbing set
consists in showing that the dynamics can be restricted to uniformly bounded solutions. To put it
in different words, we aim to show the existence of an absorbing set $B_\infty\subset L^\infty \cap H^1$.

\begin{theorem}\label{thm:absLinf}
There exists $c_0>0$ a universal constant, such that the set
\begin{equation*}
B_{\infty}=\left\{\varphi\in L^\infty\cap H^1: \|\varphi\|_{L^\infty}\leq \frac{2}{c_0\kappa}\|f\|_{L^\infty}\right\}
\end{equation*}
is an absorbing set for $S(t)$. Moreover, 
\begin{equation}\label{eq:estimateLinf}
\sup_{t\geq 0}\sup_{\theta_0\in B_\infty}\|S(t)\theta_0\|_{L^\infty}\leq \frac{3}{c_0\kappa}\|f\|_{L^\infty}.
\end{equation}
\end{theorem}
The above theorem is a consequence of Theorem \ref{thm:Linfest} below.
It is worth noticing that at this stage $B_{\infty}$ is an unbounded in $H^1$.

The main result of this section is an $L^\infty$ estimate on the solutions to \eqref{eq:SQG} based
on a De Giorgi type iteration procedure and standard a priori estimates.
The proof closely follows that of \cite{CV10a} and \cite{CD14}.

\begin{theorem}\label{thm:Linfest}
Let $\theta(t)$ be the solution to \eqref{eq:SQG} with initial datum $\theta_0\in H^1$. Then 
\begin{equation}\label{eq:linfty}
\| \theta(t)\|_{L^\infty}\leq  \frac{c}{\kappa}\left[\|\theta_0\|_{L^2}+\frac{1}{\kappa^{1/2}}\|f\|_{L^2}\right]\e^{-c_0\kappa t}+\frac{1}{c_0\kappa}\|f\|_{L^\infty}
\end{equation}
for all $t \geq 1$, for some constant $c>0$.
\end{theorem}

\begin{proof}[Proof of Theorem~\ref{thm:Linfest}]
We split the proof in two steps: first, we prove that $\theta(\tau)\in L^\infty$ for almost every $\tau\in(1/2,1)$, and then we will exploit
\eqref{eq:expdecayLinf} to conclude the proof.

For $M\geq 2\|f\|_{L^\infty}$ to be fixed later, we denote by $\eta_k$ the levels
$$
\eta_k=M(1-2^{-k})
$$
and  by $\theta_k$ the truncated function
$$
\theta_k(t)=(\theta(t)-\eta_k)_+ = \max\{ \theta(t) - \eta_k,0\}.
$$
Define also the time cutoffs 
\begin{equation}\label{eq:timecut}
\tau_k=\frac12(1-2^{-k}).
\end{equation}
As observed in \cites{CV10a, CD14}, in view of pointwise inequality of \cite{CC04}, the level set inequality
\begin{equation*}
\|\theta_k(t_2)\|^2_{L^2}+ 2\kappa \int_{t_1}^{t_2}\|\Lambda^{1/2}\theta_k(\tau)\|^2_{L^2}\dd \tau\leq \|\theta_k(t_1)\|^2_{L^2}+
2\|f\|_{L^\infty}\int_{t_1}^{t_2}\|\theta_k(\tau)\|_{L^1}\dd \tau
\end{equation*}
holds for any $t_2\geq t_1\geq 0$. Taking $t_1=s\in (\tau_{k-1},\tau_k)$ and $t_2=t\in(\tau_k,1]$, we then obtain
\begin{equation*}
\sup_{t\in[\tau_k,1]}\|\theta_k(t)\|^2_{L^2}+ 2\kappa \int_{\tau_k}^{1}\|\Lambda^{1/2}\theta_k(\tau)\|^2_{L^2}\dd \tau\leq \|\theta_k(s)\|^2_{L^2}+
2\|f\|_{L^\infty}\int_{\tau_{k-1}}^{1}\|\theta_k(\tau)\|_{L^1}\dd \tau.
\end{equation*}
Upon averaging over $s\in (\tau_{k-1},\tau_k)$, it follows that the quantity
$$
Q_k= \sup_{t\in[\tau_k,1]} \|\theta_k(t)\|^2_{L^2}+ 2\kappa \int_{\tau_k}^{1}\|\Lambda^{1/2}\theta_k(t)\|^2_{L^2}\dd t,
$$
obeys the inequality
\begin{equation}\label{eq:iter1}
Q_k\leq 2^{k}\int_{\tau_{k-1}}^{1}\|\theta_k(s)\|^2_{L^2}\dd s+2\|f\|_{L^\infty}\int_{\tau_{k-1}}^{1}\|\theta_k(t)\|_{L^1}\dd t.
\end{equation}
for all $k\in \N$, Moreover, due to \eqref{eq:energyin}, we also have
\begin{equation}\label{eq:iter00}
Q_0\leq \|\theta_0\|^2_{L^2}+\frac{1}{c_0\kappa}\|f\|^2_{L^2}.
\end{equation}
We now bound the right hand side by a power of $Q_{k-1}$.
By the H\"older inequality and the Sobolev embedding $H^{1/2}\subset L^4$, it is not hard to see that
\begin{equation}\label{eq:L3}
\|\theta_\ell\|_{L^3(\TT\times[\tau_\ell,1])}^2\leq \frac{c}{\kappa^{2/3}}Q_\ell, \qquad \forall \ell\in \N.
\end{equation}
Since 
$$
\theta_{k-1}\geq 2^{-k}M, \quad  \mbox{on the set}\quad \{(x,t):\,\theta_k(x,t)>0\},
$$
we deduce that
$$
\mathds{1}_{\{\theta_k>0\}}\leq \frac{2^k}{M}\theta_{k-1}.
$$
Using the fact that $\theta_k\leq\theta_{k-1}$ and that the bound \eqref{eq:L3} holds, we infer that
\begin{align*}
2^{k}\int_{\tau_{k-1}}^{1}\|\theta_k(s)\|^2_{L^2}\dd s 
&\leq   2^{k}\int_{\tau_{k-1}}^{1}\int_{\TT}\theta^2_{k-1}(x,s) \mathds{1}_{\{\theta_k>0\}}\dd x\,\dd s\\
&\leq   \frac{2^{2k}}{M}\int_{\tau_{k-1}}^{1}\int_{\TT}\theta^3_{k-1}(x,s) \dd x\,\dd s \leq  c\frac{2^{2k}}{M\kappa }Q_{k-1}^{3/2},
\end{align*}
and similarly,
\begin{align*}
\int_{\tau_{k-1}}^{1}\|\theta_k(t)\|_{L^1}\dd t &\leq  \int_{\tau_{k-1}}^{1}\int_{\TT}\theta_{k-1}(x,s) \mathds{1}^2_{\{\theta_k>0\}}\dd x\,\dd s\\
&\leq   \frac{2^{2k}}{M^2}\int_{\tau_{k-1}}^{1}\int_{\TT}\theta^3_{k-1}(x,s) \dd x\,\dd s \leq  c\frac{2^{2k}}{M^2 \kappa}Q_{k-1}^{3/2}.
\end{align*}
From \eqref{eq:iter1}, the above estimates and the fact that $M\geq 2\|f\|_{L^\infty}$, it follows that
\begin{equation}\label{eq:iter2}
Q_k\leq  c\frac{2^{2k}}{M\kappa }Q_{k-1}^{3/2}.
\end{equation}
Hence, if we ensure
$$
M\geq \frac{c}{\kappa}\sqrt{Q_0},
$$
then $Q_k\to 0$ as $k\to \infty$. In light of \eqref{eq:iter00}, the above constraint is in particular satisfied if
\begin{equation}\label{eq:iter3}
M\geq \frac{c}{\kappa}\left[\|\theta_0\|_{L^2}+\frac{1}{\kappa^{1/2}}\|f\|_{L^2}\right].
\end{equation}
This implies that $\theta$ is bounded above by $M$. Applying the same argument to $-\theta$, we
infer the bound
\begin{equation*}
\|\theta(\tau)\|_{L^\infty}\leq \frac{c}{\kappa}\left[\|\theta_0\|_{L^2}+\frac{1}{\kappa^{1/2}}\|f\|_{L^2}\right], \qquad \text{a.e.}\ \tau\in (1/2,1).
\end{equation*}
Once $\theta(\tau)\in L^\infty$ for some $\tau\in (1/2,1)$, we can exploit the decay estimate 
 \eqref{eq:expdecayLinf} to deduce the uniform bound \eqref{eq:linfty}, thereby concluding
 the proof.
\end{proof}
\begin{proof}[Proof of Theorem~\ref{thm:absLinf}]
Define
\begin{equation*}
B_{\infty}=\left\{\varphi\in L^\infty\cap H^1: \|\varphi\|_{L^\infty}\leq \frac{2}{c_0\kappa}\|f\|_{L^\infty}\right\}.
\end{equation*}
For a fixed bounded set $B\subset H^1$ let 
$$
R=\|B\|_{H^1}=\sup_{\varphi\in B}\|\varphi\|_{H^1}.
$$
Thanks to \eqref{eq:linfty} and the Poincar\'e inequality,
we deduce that if $\theta_0\in B$ then
\begin{equation*}
\| S(t)\theta_0\|_{L^\infty}\leq  \frac{c}{\kappa}\left[R+\frac{1}{\kappa^{1/2}}\|f\|_{L^2}\right]\e^{-c_0\kappa t}+\frac{1}{c_0\kappa}\|f\|_{L^\infty} , \qquad \forall t\geq 1.
\end{equation*}
Define the entering time $t_B = t_B(R,\|f\|_{L^2 \cap L^\infty})\geq 1$ so that
$$
\frac{c}{\kappa}\left[R+\frac{1}{\kappa^{1/2}}\|f\|_{L^2}\right]\e^{-c_0\kappa t_B}\leq\frac{1}{c_0\kappa}\|f\|_{L^\infty},
$$
for which we see that $S(t)B\subset B_\infty$ for all $t\geq t_B$. Thus $B_\infty$ is absorbing, and 
Theorem \ref{thm:absLinf} is proven.
\end{proof}

\begin{remark}\label{rmk:L2toLinf}
If we replace the time cutoffs in \eqref{eq:timecut} with 
$$
\tau_k=t_0 (1-2^{-k}), \qquad t_0\in (0,1),
$$
it follows that the solution regularizes from $L^2$ to $L^\infty$ instantaneously.
\end{remark}

\section{Nonlinear lower bounds yield  H\"older absorbing sets}
We devote this section to the improvement of the regularity of absorbing sets, namely 
from $L^\infty$ to $C^\alpha$, for $\alpha\in (0,1)$ small enough depending on $B_\infty$.

\begin{theorem}\label{thm:Calpha}
There exists 
$\alpha=\alpha(\|f\|_{L^\infty},\kappa)\in (0,1/4]$ and a constant $c_1\geq 1$
such that the set
\begin{equation*}
B_{\alpha}=\left\{\varphi\in C^\alpha\cap H^1: \|\varphi\|_{C^\alpha}\leq \frac{c_1}{\kappa}\|f\|_{L^\infty}\right\}
\end{equation*}
is an absorbing set for $S(t)$. Moreover, 
\begin{equation}\label{eq:Calphaunif}
\sup_{t\geq 0}\sup_{\theta_0\in B_\alpha}\|S(t)\theta_0\|_{C^\alpha}\leq \frac{2 c_1}{\kappa}\|f\|_{L^\infty},
\end{equation}
holds.
\end{theorem}
In light of Theorem \ref{thm:absLinf}, the solutions to \eqref{eq:SQG} emerging from data in a bounded subset of $H^1$ are absorbed in finite time
by a fixed subset of $L^\infty$. Therefore, in order to prove Theorem \ref{thm:Calpha}, it suffices
to restrict our attention to solutions emanating from initial data $\theta_0\in L^\infty$ 
and derive a number of a priori bounds solely in terms of $\|\theta_0\|_{L^\infty}$. 
For convenience, in the course of this section we will set
\begin{equation}\label{eq:globalLinf}
K_\infty=\|\theta_0\|_{L^\infty}+\frac{1}{c_0\kappa}\|f\|_{L^\infty},
\end{equation}
so that in view of \eqref{eq:expdecayLinf} the solution originating from $\theta_0$ satisfies the global bound
\begin{equation}\label{eq:globalLinf2}
\|\theta(t)\|_{L^\infty}\leq K_\infty, \qquad \forall t\geq 0.
\end{equation}
The main result of this section is the following a priori estimate in suitable H\"older space.
\begin{theorem}\label{thm:Calphaest}
Assume that $\theta_0\in L^\infty\cap H^1$. There exists 
$\alpha=\alpha(\|\theta_0\|_{L^\infty},\|f\|_{L^\infty},\kappa)\in (0,1/4]$
such that
\begin{equation}\label{eq:Calpha}
\|\theta(t)\|_{C^\alpha}\leq c \left[\|\theta_0\|_{L^\infty}+\frac{1}{c_0\kappa}\|f\|_{L^\infty}\right], \qquad \forall t\geq t_\alpha= \frac{3}{2(1-\alpha)}
\end{equation} 
for some positive constant $c>0$.
\end{theorem}
The precise expression of $\alpha$ is given below in \eqref{eq:alpha}. The proof of Theorem \ref{thm:Calphaest}
requires several intermediate steps culminating in Lemma \ref{lem:Calpha}. For now, let us prove
Theorem \ref{thm:Calpha} assuming Theorem \ref{thm:Calphaest}.

\begin{proof}[Proof of Theorem \ref{thm:Calpha}]
We first show that there exists $\alpha\in(0,1/4]$ and $c_1\geq 1$ such that $B_\alpha$ is absorbing. Clearly,
it is enough to prove that the $L^\infty$-absorbing set $B_\infty$ is itself absorbed by $B_\alpha$.
Fix $\alpha$ as suggested by  Theorem \ref{thm:Calphaest}, namely,
$$
\alpha=\alpha(\|B_\infty\|_{L^\infty},\|f\|_{L^\infty},\kappa), \quad \mbox{where} \quad \|B_\infty\|_{L^\infty}=\sup_{\varphi\in B_\infty}\|\varphi\|_{L^\infty}\leq \frac{2}{c_0\kappa}\|f\|_{L^\infty}.
$$
Take $\theta_0\in B_\infty$. By \eqref{eq:estimateLinf}, 
$$
\|S(t)\theta_0\|_{L^\infty}\leq \frac{3}{c_0\kappa}\|f\|_{L^\infty}, \qquad \forall t\geq 0.
$$
Consequently, \eqref{eq:Calpha} implies that
$$
\|S(t)\theta_0\|_{C^\alpha}\leq \frac{4c}{c_0\kappa}\|f\|_{L^\infty}, \qquad \forall t\geq t_\alpha,
$$
namely $S(t)\theta_0\in B_\alpha$ for all $t\geq t_\alpha$, upon choosing $c_1=4c/c_0$. The fact
that $t_\alpha$ depends only on $\|B_\infty\|_{L^\infty}$, $\|f\|_{L^\infty}$, and $\kappa$, implies that
$$
S(t)B_\infty\subset B_\alpha, \qquad \forall t\geq t_\alpha,
$$
as sought. The uniform estimate \eqref{eq:Calphaunif} follows from the propagation 
of H\"older regularity proven in \cite{CTV13}, namely the property that if $\theta_0\in C^\alpha$,
then 
\begin{equation}\label{eq:propag}
\|S(t)\theta_0\|_{C^\alpha}\leq [\theta_0]_{C^\alpha}+c \left[\|\theta_0\|_{L^\infty}+\frac{1}{c_0\kappa}\|f\|_{L^\infty}\right], \qquad \forall t\geq0.
\end{equation}
This concludes the proof of the theorem.
\end{proof}

The rest of the section is dedicated to the proof of Theorem \ref{thm:Calphaest}. The techniques employed
have the flavor of those devised in \cites{CTV13,CV12}, although the approach is closely related to that of \cite{CZV14}
used for a proof of eventual regularity for supercritical SQG.

\subsection{Time dependent nonlinear lower bounds}
In order to estimate $C^\alpha$-seminorms it is natural to consider  the finite difference
\begin{align*}
\delta_h\theta(x,t)=\theta(x+h,t)-\theta(x,t),
\end{align*}
which is periodic in both $x$ and $h$, where  $x,h \in \TT$.  As in \cites{CV12,CTV13}, it follows that
\begin{equation}\label{eq:findiff}
L (\delta_h\theta)^2+ D[\delta_h\theta]=0,
\end{equation}
where $L$ denotes the differential operator
\begin{equation}
\label{eq:L:def}
L=\de_t+\u\cdot \nabla_x+(\delta_h\u)\cdot \nabla_h+ \Lambda
\end{equation}
and 
\begin{equation}
\label{eq:D:gamma:def}
D[\varphi](x)= c \int_{\R^2} \frac{\big[\varphi(x)-\varphi(x+y)\big]^2}{|y|^{3}}\dd y.
\end{equation}
Let $\xi:[0,\infty)\to[0,\infty)$ be a bounded decreasing differentiable function to be determined later. For 
\begin{equation*}
0<\alpha \leq \frac14
\end{equation*}
to be fixed later on,
we study the evolution of the quantity $v(x,t;h)$ defined by
\begin{equation}\label{eq:v}
v(x,t;h) =\frac{|\delta_h\theta(x,t)|}{(\xi(t)^2+|h|^2)^{\alpha/2}}.
\end{equation}
The main point is that when $\xi(t)=0$ we have that 
$$
\|v(t)\|_{L^\infty_{x,h}}=\esup_{x,h\in\TT} |v(x,t;h)| = \sup_{x\neq y \in\TT} \frac{|\theta(x,t)-\theta(y,t)|}{|x-y|^{\alpha}} = [\theta(t)]_{C^\alpha}.
$$
From \eqref{eq:findiff} and a short calculation (see~\cite{CZV14}) we obtain that
\begin{align}
L v^2+\frac{\kappa D[\delta_h\theta] }{(\xi^2+|h|^2)^\alpha}
&=2\alpha |\dot\xi|\frac{\xi}{\xi^2+|h|^2}v^2 -2\alpha \frac{h}{\xi^2+|h|^2}\cdot \delta_h\u \, v^2+
\frac{\delta_hf}{(\xi^2+|h|^2)^{\alpha/2}} v\notag \\
&\leq  2\alpha |\dot\xi|\frac{\xi}{\xi^2+|h|^2}v^2 +2\alpha \frac{|h|}{\xi^2+|h|^2}|\delta_h\u|v^2+
\frac{2\|f\|_{L^\infty}}{(\xi^2+|h|^2)^{\alpha/2}} v\label{eq:ineq1}
\end{align}
where $\delta_h\u= \RR^\perp \delta_h\theta$. We will bound the terms on the right-hand
side of \eqref{eq:ineq1} in such a way so that they can be compared with the dissipative term $D[\delta_h\theta]$
and its nonlinear lower bounds derived in the following lemma.

\begin{lemma}\label{lem:nonlinbdd}
There exists a positive constant
$c_2$ such that
\begin{equation}\label{eq:N1}
\frac{D[\delta_h\theta](x,t)}{(\xi(t)^2+|h|^2)^\alpha}
\geq \frac{|v(x,t;h)|^3}{c_2\|\theta(t)\|_{L^\infty}(\xi(t)^2+|h|^2)^{\frac{1-\alpha}{2}}} 
\end{equation}
holds for any $x,h\in \TT$ and any $t\geq 0$.
\end{lemma}

\begin{proof}[Proof of Lemma~\ref{lem:nonlinbdd}]
In the course of the proof, we omit the dependence on $t$ of all functions. It is understood
that every calculation is performed pointwise in $t$.
Arguing as in \cite{CTV13}, it can be shown that for $r\geq 4|h|$ there holds
$$
D[\delta_h\theta](x)\geq \frac{1}{2r}|\delta_h\theta(x)|^2-c|\delta_h\theta(x)|\|\theta\|_{L^\infty}\frac{|h|}{r^2},
$$
where $c\geq 1$ is an absolute constant.
A choice satisfying $r\geq 4(\xi^2+|h|^2)^{1/2}\geq 4|h|$ can be made as
$$
r=\frac{4c\|\theta\|_{L^\infty}}{|\delta_h\theta(x)|} (\xi^2+|h|^2)^{1/2},
$$
from which it follows that
\begin{align*}
D[\delta_h\theta](x)&\geq \frac{|\delta_h\theta(x)|^2}{2r}\left[1-\frac12\frac{|h|}{(\xi^2+|h|^2)^{1/2}}\right]\\
&\geq  \frac{|\delta_h\theta(x)|^2}{4r}=\frac{|\delta_h\theta(x)|^3}{16c \|\theta\|_{L^\infty} (\xi^2+|h|^2)^{1/2}}
\end{align*}
The lower bound \eqref{eq:N1} follows by dividing the above inequality by $(\xi^2+|h|^2)^\alpha$.
\end{proof}
The choice for the function $\xi$ is now closely related to the lower bound \eqref{eq:N1}. We assume that $\xi$
solves the ordinary differential equation
\begin{equation}\label{eq:ODE}
\dot\xi =- \xi^{{\frac{1+2\alpha}{3}}},\qquad \xi(0)=1.
\end{equation}
More explicitly,
\begin{equation}\label{eq:xi}
\xi(t)=\begin{cases}
\displaystyle \left[1-\frac23(1-\alpha) t\right]^{\frac{3}{2(1-\alpha)}}, \quad &\text{if } t\in [0,t_\alpha],\\ \\
0,\quad &\text{if } t \in  (t_\alpha,\infty),
\end{cases}
\end{equation}
where 
\begin{equation}\label{eq:regtime}
t_\alpha=\frac{3}{2(1-\alpha)}.
\end{equation}
We then have the following result.
\begin{lemma}\label{lem:ODE}
Assume that the function $\xi:[0,\infty)\to[0,\infty)$ is given by \eqref{eq:xi}. Then
 the estimate
\begin{equation}\label{eq:est1}
2\alpha |\dot\xi(t)|\frac{\xi(t)}{\xi(t)^2+|h|^2}|v(x,t;h)|^2\leq \frac{\kappa|v(x,t;h)|^3}{8c_2\|\theta(t)\|_{L^\infty}(\xi(t)^2+|h|^2)^{\frac{1-\alpha}{2}}}  +\frac{c}{\kappa^2}\|\theta(t)\|^2_{L^\infty}
\end{equation}
holds pointwise for $x,h \in \TT$ and $t\geq 0$, where $c_2$ is the same constant appearing in \eqref{eq:N1}.
\end{lemma}
\begin{proof}[Proof of Lemma~\ref{lem:ODE}]
We again suppress the $t$-dependence in all the estimates below. 
In view of \eqref{eq:ODE} and the fact that $\alpha\leq 1/4$, a simple computation shows that
$$
2\alpha |\dot\xi|\frac{\xi}{\xi^2+|h|^2}|v(x;h)|^2\leq \frac12\frac{\xi^{{\frac{4+2\alpha}{3}}}}{\xi^2+|h|^2}|v(x;h)|^2
\leq  \frac12\frac{|v(x;h)|^2}{(\xi^2+|h|^2)^{\frac{1-\alpha}{3}}}.
$$
Therefore, the $\eps$-Young inequality
\begin{equation}\label{eq:young}
ab\leq \frac{2\eps}{3} a^{3/2}+\frac{1}{3\eps^2}b^3,\qquad a,b,\eps>0
\end{equation}
with $\eps=\kappa/(12c_2\|\theta\|_{L^\infty})$ implies that
$$
2\alpha |\dot\xi|\frac{\xi}{\xi^2+|h|^2}|v(x;h)|^2
\leq\frac{\kappa|v(x;h)|^3}{8c_2\|\theta\|_{L^\infty}(\xi^2+|h|^2)^{\frac{1-\alpha}{2}}}
+\frac{c}{\kappa^2}\|\theta\|^2_{L^\infty},
$$
which is what we claimed.
\end{proof}
In the same fashion, we can estimate the forcing term appearing in \eqref{eq:ineq1}.

\begin{lemma}\label{lem:force}
For every $x,h \in \TT$ and $t\geq 0$ we have
\begin{equation}\label{eq:est2}
\frac{2\|f\|_{L^\infty}}{(\xi(t)^2+|h|^2)^{\alpha/2}} v(x,t;h)\leq 
\frac{\kappa|v(x,t;h)|^3}{8c_2\|\theta(t)\|_{L^\infty}(\xi(t)^2+|h|^2)^{\frac{1-\alpha}{2}}}
+ c\kappa^{1/2} \|f\|_{L^\infty}^{3/2}\|\theta(t)\|^{1/2}_{L^\infty},
\end{equation}
where $c_2$ is the same constant appearing in \eqref{eq:N1}.
\end{lemma}

\begin{proof}[Proof of Lemma~\ref{lem:force}]
Applying once more Young inequality \eqref{eq:young}
we infer that
$$
\frac{2\|f\|_{L^\infty}}{(\xi^2+|h|^2)^{\alpha/2}} v(x;h)\leq 
\frac{\kappa|v(x;h)|^3}{8c_2\|\theta\|_{L^\infty}(\xi^2+|h|^2)^{\frac{1-\alpha}{2}}}+c (\xi^2+|h|^2)^{\frac{1-4\alpha}{4}}
\kappa^{1/2}\|f\|_{L^\infty}^{3/2}\|\theta\|^{1/2}_{L^\infty}.
$$
The conclusion follows from the assumption $\alpha\leq 1/4$ and the bounds $\xi,|h|\leq 1$.
\end{proof}
If we now apply the bounds \eqref{eq:est1}-\eqref{eq:est2} to \eqref{eq:ineq1}, we end up with
\begin{equation}\label{eq:ineq2}
\begin{aligned}
L v^2+\frac{\kappa}{2}\frac{ D[\delta_h\theta] }{(\xi^2+|h|^2)^\alpha}&+\frac{\kappa|v|^3}{4c_2\|\theta\|_{L^\infty}(\xi^2+|h|^2)^{\frac{1-\alpha}{2}}}\\
&\leq 2\alpha \frac{|h|}{\xi^2+|h|^2}|\delta_h\u|v^2+c\left[\|\theta\|^2_{L^\infty}+\kappa^{1/2} \|f\|_{L^\infty}^{3/2}\|\theta\|^{1/2}_{L^\infty}\right].
\end{aligned}
\end{equation}
In the next section, we provide an upper bound on the remaining term containing $\delta_h\u$. 

\subsection{Estimates on the nonlinear term}
We would like to stress once more that the only restriction on $\alpha$ so far has consisted in
imposing $\alpha\in(0,1/4]$. This arose only in the proof of Lemma \ref{lem:force}.
In order to deal with Riesz-transform contained in $\delta_h\u$, the H\"older exponent will be
further restricted in terms of the initial datum $\theta_0$ and the forcing term $f$. It is crucial
that this restriction only depends on $\|\theta_0\|_{L^\infty}$ and $\|f\|_{L^\infty}$.
\begin{lemma}\label{lem:rieszbdd}
Suppose that $\theta_0\in L^\infty$, and set 
\begin{equation}\label{eq:alpha}
\alpha=\min\left\{\frac{\kappa}{c_3K_\infty},\frac14\right\}, \qquad K_\infty=\|\theta_0\|_{L^\infty}+\frac{1}{c_0\kappa}\|f\|_{L^\infty},
\end{equation}
for a universal constant $c_3\geq 64$. Then
\begin{equation}\label{eq:est3}
2\alpha \frac{|h||\delta_h\u(x,t)|}{\xi(t)^2+|h|^2}|v(x,t;h)|^2\leq  \frac{\kappa}{2} \frac{D[\delta_h\theta](x,t)}{(\xi(t)^2+|h|^2)^\alpha}
+ \frac{\kappa}{8c_2K_\infty(\xi(t)^2+|h|^2)^{\frac{1-\alpha}{2}}} |v(x,t;h)|^3,
\end{equation}
for every $x,h \in \TT$ and $t\geq 0$, where $c_2$ is the same constant appearing in \eqref{eq:N1}.
\end{lemma}

\begin{proof}[Proof of Lemma~\ref{lem:rieszbdd}]
By the same arguments of \cites{CTV13,CV12}, for $r\geq 4|h|$ it is possible to derive the upper bound
$$
|\delta_h\u(x)|\leq c \left[r^{1/2} \big[D[\delta_h\theta](x)\big]^{1/2} +\frac{\|\theta\|_{L^\infty}|h|}{r} \right],
$$
pointwise in $x,h\in \TT$ and $t\geq 0$. Using the Cauchy-Schwarz inequality, we deduce that
\begin{align*}
\frac{2\alpha  |h|}{\xi^2+|h|^2}|\delta_h\u(x)||v(x;h)|^2
&\leq \frac{2\alpha}{(\xi^2+|h|^2)^{1/2}}|\delta_h\u(x)||v(x;h)|^2\\
&\leq \frac{\kappa}{2} \frac{D[\delta_h\theta](x)}{(\xi^2+|h|^2)^\alpha} + c\left[\frac{\alpha^2}{\kappa(\xi^2+|h|^2)^{1-\alpha}}r |v(x;h)|^4 +\alpha \frac{\|\theta\|_{L^\infty}}{r}|v(x;h)|^2\right].
\end{align*}
We then choose $r$ as
$$
r=\frac{\kappa^{1/2}\|\theta\|^{1/2}_{L^\infty} (\xi^2+|h|^2)^{\frac{1-\alpha}{2}}}{\alpha^{1/2}v(x;h)}=
\frac{\kappa^{1/2}\|\theta\|^{1/2}_{L^\infty}(\xi^2+|h|^2)^{1/2}}{\alpha^{1/2}|\delta_h\theta(x)|}.
$$
In view of \eqref{eq:alpha}, this is a feasible choice, since 
$$
r\geq \frac{\kappa^{1/2}\|\theta\|^{1/2}_{L^\infty}}{2\alpha^{1/2}\|\theta\|_{L^\infty}}|h|=
\frac{\kappa^{1/2}}{2\alpha^{1/2}\|\theta\|^{1/2}_{L^\infty}}|h|\geq 
\frac{\kappa^{1/2}}{2\alpha^{1/2}K_\infty^{1/2}}|h|\geq 4|h|.
$$
Thus, thanks to \eqref{eq:alpha}, we obtain
\begin{align*}
2\alpha \frac{|h|}{\xi^2+|h|^2}|\delta_h\u(x)||v(x;h)|^2
&\leq \frac{\kappa}{2} \frac{D[\delta_h\theta](x)}{(\xi^2+|h|^2)^\alpha}
+ c\frac{\alpha^{3/2}\|\theta\|^{1/2}_{L^\infty} }{\kappa^{1/2}(\xi^2+|h|^2)^{\frac{1-\alpha}{2}}} |v(x;h)|^3\\
&\leq \frac{\kappa}{2} \frac{D[\delta_h\theta](x)}{(\xi^2+|h|^2)^\alpha}
+ c\frac{\alpha^{3/2}K_\infty^{1/2} }{\kappa^{1/2}(\xi^2+|h|^2)^{\frac{1-\alpha}{2}}} |v(x;h)|^3\\
&\leq \frac{\kappa}{2} \frac{D[\delta_h\theta](x)}{(\xi^2+|h|^2)^\alpha}
+ c\frac{\alpha}{(\xi^2+|h|^2)^{\frac{1-\alpha}{2}}} |v(x;h)|^3.
\end{align*}
By possibly further reducing $\alpha$ so that
$$
\alpha \leq \frac{\kappa}{8c c_2 K_\infty},
$$
we deduce that
\begin{align*}
2\alpha \frac{|h|}{\xi^2+|h|^2}|\delta_h\u(x)||v(x;h)|^2
\leq \frac{\kappa}{2} \frac{D[\delta_h\theta](x)}{(\xi^2+|h|^2)^\alpha}
+ \frac{\kappa}{8c_2K_\infty(\xi^2+|h|^2)^{\frac{1-\alpha}{2}}} |v(x;h)|^3,
\end{align*}
which concludes the proof.
\end{proof}

We now proceed with the last step in the proof of Theorem \ref{thm:Calphaest}, which consists
of H\"older $C^\alpha$ estimates, where the exponent $\alpha$ is given by \eqref{eq:alpha}.

\subsection{Locally uniform H\"older estimates}
From the global bound \eqref{eq:globalLinf}, \eqref{eq:ineq2} and the estimate \eqref{eq:est3},
it follows that for $\alpha$ complying with  \eqref{eq:alpha} the function $v^2$ satisfies
\begin{equation*}
L v^2+\frac{\kappa|v|^3}{8c_2K_\infty(\xi^2+|h|^2)^{\frac{1-\alpha}{2}}}
\leq c\left[K_\infty^2+\kappa^{1/2} \|f\|_{L^\infty}^{3/2}K_\infty^{1/2}\right].
\end{equation*}
Taking into account that $\xi^2+|h|^2\leq 1 + {\rm diam}(\TT)^2 = 2$ for all $h \in \TT$, and that $\|f\|_{L^\infty}\leq c_0\kappa K_{\infty}$, we arrive at
\begin{equation}\label{eq:ineq3}
L v^2+\frac{\kappa|v|^3}{16c_2K_\infty}
\leq cK_\infty^2
\end{equation} 
which holds pointwise for $(x,h) \in \TT \times \TT$.
In the next lemma we show that the above inequality gives uniform control on the $C^\alpha$ seminorm
of the solution.

\begin{lemma}\label{lem:Calpha}
Assume that $\theta_0\in L^\infty$, and fix $\alpha$ as in \eqref{eq:alpha}. There exists a time $t_\alpha>0$
such that the solution to \eqref{eq:SQG} with initial datum $\theta_0$ is $\alpha$-H\"older continuous. Specifically,
\begin{equation*}
[\theta(t)]_{C^\alpha}\leq c \left[\|\theta_0\|_{L^\infty}+\frac{1}{c_0\kappa}\|f\|_{L^\infty}\right], \qquad \forall t\geq t_\alpha= \frac{3}{2(1-\alpha)}.
\end{equation*}
\end{lemma}

\begin{proof}[Proof of Lemma~\ref{lem:Calpha}]
Thanks to \eqref{eq:ineq3}, the function
$$
\psi(t)=\|v(t)\|_{L^\infty_{x,h}}^2
$$
satisfies the differential inequality
\begin{equation}\label{eq:diffineq}
\frac{\dd}{\dd t} \psi+\frac{\kappa}{16c_2K_\infty} \psi^{3/2}\leq cK_\infty^2.
\end{equation}
This can be justified as follows: $v^2$ is a bounded continuous function of $x$ and $h$, so that we 
can evaluate \eqref{eq:ineq3} at a point
$(\bar x, \bar h) = (\bar x(t), \bar h(t)) \in \TT\times \TT$ at which $v^2(t)$ attains its maximum value.
Since, at this point we have $\de_hv^2=\de_xv^2=0$ and $\Lambda v^2\geq 0$, the inequality \eqref{eq:diffineq} holds in view of the Rademacher theorem (see
\cites{CC04,CTV13, CZV14} for details). 
Moreover, by the very definition of $v$,
$$
\psi(0)\leq\frac{4\|\theta_0\|^2_{L^\infty}}{\xi_0^{2\alpha}}=4\|\theta_0\|^2_{L^\infty}\leq 4K_\infty^2.
$$
From a standard comparison for ODEs it immediately follows that 
\begin{equation}\label{eq:bddpsi}
\psi(t)\leq cK^2_\infty, \qquad \forall t\geq 0,
\end{equation}
for some sufficiently large constant $c>0$.
With \eqref{eq:bddpsi} at hand, we have thus proven that 
$$
[\theta(t)]_{C^\alpha}^2=\psi(t) \leq c K_\infty^2, \qquad \forall t\geq t_\alpha,
$$
where $t_\alpha$ is given by \eqref{eq:regtime}, thereby concluding the proof.
\end{proof}

\begin{proof}[Proof of Theorem \ref{thm:Calphaest}]
The bound \eqref{eq:Calpha}
follows from estimate \eqref{eq:globalLinf2} for the $L^\infty$ norm and the bound of  Lemma~\ref{lem:Calpha} for the H\"older seminorm
\begin{align*}
\|\theta(t)\|_{C^\alpha}&=\|\theta(t)\|_{L^\infty}+[\theta(t)]_{C^\alpha} \leq K_\infty+[\theta(t)]_{C^\alpha} \leq c K_\infty
\qquad \forall t\geq 0,
\end{align*}
and a sufficiently large constant $c>0$.
\end{proof}

\begin{remark}
The quantitive regularization estimate at time $t_\alpha$ from $L^\infty$ to $C^\alpha$ is given by the ODE \eqref{eq:ODE}.
More precisely, $t_\alpha$ is determined by the initial datum $\xi(0)$, conveniently chosen to be 1 in the proof above. If instead we let $\xi(0)=\xi_0>0$,
then
\begin{equation}\label{eq:xi2}
\xi(t)=\begin{cases}
\displaystyle \left[\xi_0^{\frac{2(1-\alpha)}{3}}-\frac23(1-\alpha) t\right]^{\frac{3}{2(1-\alpha)}}, \quad &\text{if } t\in [0,t_\alpha],\\ \\
0,\quad &\text{if } t \in  (t_\alpha,\infty),
\end{cases}
\end{equation}
and 
\begin{equation}\label{eq:regtime2}
t_\alpha=\frac{3 }{2(1-\alpha)}\xi_0^{\frac{2(1-\alpha)}{3}}.
\end{equation}
In particular, $t_\alpha$ can be made arbitrarily small by a suitable small choice of $\xi_0$. This observation,
together with Remark~\ref{rmk:L2toLinf} recovers the result of~\cite{CV10a} and shows that solutions to \emph{forced}
 \eqref{eq:SQG} regularize instantaneously from
$L^2$ to $C^\alpha$.
\end{remark}

\section{The absorbing set in \texorpdfstring{$H^1$}{H 1}}
With Theorem \ref{thm:Calpha} at hand, it is now possible to ensure the existence of a bounded
absorbing set in $H^1$.

\begin{theorem}\label{thm:H1abs}
There exists $\alpha=\alpha(\|f\|_{L^\infty},\kappa)\in (0,1/4]$ 
and a constant $R_1=R_1(\|f\|_{L^\infty\cap H^1}, \kappa)\geq 1$
such that the set
\begin{equation*}
B_1=\left\{\varphi\in C^\alpha\cap H^1: \|\varphi\|^2_{H^1}+\|\varphi\|^2_{C^\alpha}
\leq R_1^2\right\}
\end{equation*}
is an absorbing set for $S(t)$. Moreover, 
\begin{equation}\label{eq:H1unif}
\sup_{t\geq 0}\sup_{\theta_0\in B_1}\left[\|S(t)\theta_0\|^2_{H^1}+\|S(t)\theta_0\|^2_{C^\alpha}+\int_t^{t+1}\|S(\tau)\theta_0\|^2_{H^{3/2}}\dd \tau\right]\leq 
2 R_1^2.
\end{equation}
The expression for $R_1$ can be computed explicitly from \eqref{eq:Calphaunif} and \eqref{eq:K1} below.
\end{theorem}
Since in establishing the existence of an $H^1$ absorbing ball the dynamics can be restricted to the $C^\alpha$ absorbing ball, 
in order to prove Theorem~\ref{thm:H1abs} it is enough to establish an a priori estimate for initial data that are H\"older continuous.

\begin{lemma}\label{lem:H1absest}
Assume that $\theta_0\in H^1\cap C^\alpha$. Then
\begin{equation}\label{eq:H1exp}
\|\theta(t)\|^2_{H^1}\leq \|\theta_0\|^2_{H^1}\e^{-\frac{c_0\kappa}{4} t}+ K_1,
\end{equation}
where $K_1=K_1(\|f\|_{L^\infty\cap H^1}, \kappa, \|\theta_0\|_{C^\alpha})\geq 1$ is given in \eqref{eq:K1} below.
Moreover, for every $t\geq 0$ we have
\begin{equation}\label{eq:H32int}
\int_t^{t+1}\|\theta(\tau)\|^2_{H^{3/2}}\dd \tau\leq \frac{c}{\kappa}\left[\|\theta_0\|^2_{H^1}+ K_1\right].
\end{equation}
\end{lemma}

\begin{proof}[Proof of Lemma~\ref{lem:H1absest}]
The proof closely follows the lines of \cite{CTV13}, and thus we omit many details.
We apply $\nabla$ to \eqref{eq:SQG} and take the inner product with $\nabla \theta$, to obtain
\begin{equation}\label{eq:gradeq}
(\de_t+\u\cdot \nabla+\Lambda)|\nabla\theta|^2+\kappa D[\nabla \theta]=-2\de_\ell\u_j\de_j\theta\de_\ell\theta+2\nabla f\cdot \nabla \theta, 
\end{equation}
pointwise in $x$, where, as before,
\begin{equation*}
D[\nabla \theta](x)= c \int_{\R^2} \frac{\big[\nabla \theta(x)-\nabla \theta(x+y)\big]^2}{|y|^{3}}\dd y.
\end{equation*}
From \eqref{eq:propag}, we also know that
\begin{equation*}
\|\theta(t)\|_{C^\alpha}\leq M:=
c \left[\|\theta_0\|_{C^\alpha}+\frac{1}{c_0\kappa}\|f\|_{L^\infty}\right], \qquad \forall t\geq0.
\end{equation*}
Thanks to \cite{CV12}*{Theorem 2.2}, we then deduce the lower bound
\begin{equation}\label{eq:nonbdd}
D[\nabla \theta](x,t)\geq \frac{|\nabla \theta(x,y)|^{\frac{3-\alpha}{1-\alpha}}}{c_4 M^\frac{1}{1-\alpha}} 
\end{equation}
Arguing as in Lemma \ref{lem:rieszbdd}, we obtain for $r>0$ that 
$$
|\nabla\u(x,t)|\leq c \left[r^{1/2} \big[D[\nabla\theta](x,t)\big]^{1/2} +\frac{M}{r} \right],
$$
By choosing $r=\kappa^{1/2}M^{1/2}|\nabla\theta(x,t)|^{-1}$ and the Cauchy-Schwarz inequality
we then infer that
$$
|\nabla\u(x,t)||\nabla\theta(x,t)|^2\leq  \frac\kappa2 D[\nabla\theta](x,t)+ \frac{c}{\kappa^{1/2}}M^{1/2}|\nabla\theta(x,t)|^3.
$$
From \eqref{eq:gradeq}, we have 
\begin{align*}
(\de_t+\u\cdot \nabla+\Lambda)|\nabla\theta|^2+\frac\kappa2 D[\nabla \theta]
&\leq \frac{c}{\kappa^{1/2}}M^{1/2}|\nabla\theta(x,t)|^3+2|\nabla f||\nabla \theta| \notag\\
& \leq  \frac\kappa4 \frac{ |\nabla \theta(x,y)|^{\frac{3-\alpha}{1-\alpha}}}{ c_4 M^\frac{1}{1-\alpha}} +
\left[\frac{c M}{\kappa}\right]^{\frac1{4\alpha}}+2|\nabla f||\nabla \theta|,
\end{align*}
so that together with \eqref{eq:nonbdd} we arrive at
\begin{equation}\label{eq:gradeq2}
(\de_t+\u\cdot \nabla+\Lambda)|\nabla\theta|^2+\frac\kappa4 D[\nabla \theta]\leq \left[\frac{c M}{\kappa}\right]^{\frac1{4\alpha}}+2|\nabla f||\nabla \theta|. 
\end{equation}
Integrating over $\TT$ and using the identity
$$
\frac12\int_{\TT} D[\nabla \varphi](x)\dd x=\int \nabla\varphi(x)\cdot \Lambda\nabla\varphi(x)\dd x=\|\varphi\|_{H^{3/2}}^2,
$$
we obtain the differential inequality
\begin{equation*}
\ddt \|\theta\|_{H^1}^2+\frac{\kappa}{2}\|\theta\|^2_{H^{3/2}}\leq \left[\frac{c M}{\kappa}\right]^{\frac1{4\alpha}}+2\| f\|_{H^1}\| \theta\|_{H^1}.
\end{equation*}
From the Poincar\'e inequality, we then have
\begin{equation}\label{eq:H1diff}
\ddt \|\theta\|_{H^1}^2+\frac{\kappa}{4}\|\theta\|^2_{H^{3/2}}\leq  \left[\frac{c M}{\kappa}\right]^{\frac1{4\alpha}}+\frac{4}{c_0\kappa}\| f\|_{H^1}^2.
\end{equation}
From the above, \eqref{eq:H1exp} follows from the Poincar\'e inequality and the standard Gronwall lemma,
provided we set 
\begin{equation}\label{eq:K1}
K_1:=\frac{4}{c_0\kappa}\left[\left(\frac{c M}{\kappa}\right)^{\frac1{4\alpha}}+\frac{4}{c_0\kappa}\| f\|_{H^1}^2
\right].
\end{equation}
By integrating \eqref{eq:H1diff} on $(t,t+1)$ and applying \eqref{eq:H1exp}, we also recover \eqref{eq:H32int}.
\end{proof}

\section{The global attractor}\label{sec:globattr}
Once the existence of an  $H^1$-bounded absorbing set for $S(t)$ is established, we aim to prove
the existence of the global attractor by improving the regularity of the absorbing set to $H^{3/2}$
(see Theorem \ref{thm:H32abs} below).
Following the general theory recently developed in \cite{CCP12}, this automatically implies the
existence of a minimal compact attracting set for $S(t)$ which, however, might not be invariant, 
due to the possible lack of continuity of $S(t)$, for fixed $t>0$, as a map acting on $H^1$ (see
\cites{CCP12,CZK15} for examples of non-invariant attractors).
Full invariance will be recovered in a subsequent step (Section \ref{sub:invariance}), by exploiting
the $H^{3/2}$-regularity of the absorbing set and a continuity estimate proven in 
\cite{CTV13}*{Proposition 5.5}.

\subsection{Compact absorbing sets}
The existence and regularity of the attractor in Theorem \ref{thm:globalattra} follow from the existence of an
absorbing set bounded in $H^{3/2}$. 

\begin{theorem}\label{thm:H32abs}
There exists a constant $R_2=R_2(\|f\|_{L^\infty\cap H^1}, \kappa)\geq 1$
such that the set
\begin{equation*}
B_2=\left\{\varphi\in H^{3/2}: \|\varphi\|_{H^{3/2}}
\leq R_2\right\}
\end{equation*}
is an absorbing set for $S(t)$. Moreover, 
\begin{equation}\label{eq:H32unif}
\sup_{t\geq 0}\sup_{\theta_0\in B_2}\|S(t)\theta_0\|_{H^{3/2}}\leq 
2 R_2.
\end{equation}
\end{theorem}

\begin{proof}[Proof of Theorem~\ref{thm:H32abs}]
As usual, it is enough to show that $B_2$ absorbs $B_1$, the $H^1$ absorbing set obtained in Theorem~\ref{thm:H1abs}. If $\theta_0\in B_1$, then
\eqref{eq:H1unif} implies that 
\begin{equation}\label{eq:H32int2}
\sup_{t\geq 0}\int_t^{t+1}\|S(\tau)\theta_0\|^2_{H^{3/2}}\dd \tau\leq 2 R^2_1.
\end{equation}
By testing \eqref{eq:SQG} with $\Lambda^3\theta$ and using standard arguments, we deduce that
\begin{align*}
\ddt \|\theta\|^2_{H^{3/2}}+\kappa\|\theta\|^2_{H^2}\leq \frac1\kappa\|f\|^2_{H^1}+
2\left|\int_{\TT}\left[\Lambda^{3/2}(\u\cdot\nabla\theta)-\u\cdot \nabla\Lambda^{3/2}\theta\right]\Lambda^{3/2}\theta \dd x\right|.
\end{align*}
By means of the commutator estimate
$$
\|\Lambda^{3/2} (\varphi \psi)-\varphi\Lambda^{3/2}\psi\|_{L^2}
\leq c\left[\|\nabla \varphi\|_{L^4}\|\Lambda^{1/2}\psi\|_{L^4}+\|\Lambda^{3/2}\varphi\|_{L^4}\|\psi\|_{L^4}\right], 
$$
and the Sobolev embedding $H^{1/2}\subset L^4$, we therefore have
\begin{align*}
\ddt \|\theta\|^2_{H^{3/2}}+\kappa\|\theta\|^2_{H^2}
&\leq  \frac1\kappa\|f\|^2_{H^1}+
c\|\theta\|_{H^{3/2}}\left[\|\Lambda \u\|_{L^4}\|\Lambda^{3/2}\theta\|_{L^4}+\|\Lambda^{3/2}\u\|_{L^4}\|\Lambda\theta\|_{L^4}\right]\\
&\leq \frac1\kappa\|f\|^2_{H^1}+c\|\theta\|_{H^{3/2}}^2\|\theta\|_{H^2}\\
&\leq \frac1\kappa\|f\|^2_{H^1}+\frac{c}{\kappa}\|\theta\|_{H^{3/2}}^4+\frac\kappa2\|\theta\|^2_{H^2}.
\end{align*}
Hence,
\begin{align*}
\ddt \|\theta\|^2_{H^{3/2}}+\frac\kappa2\|\theta\|^2_{H^2}\leq \frac1\kappa\|f\|^2_{H^1}+\frac{c}{\kappa}\|\theta\|_{H^{3/2}}^4.
\end{align*}
Thanks to the local integrability \eqref{eq:H32int2} and the above differential
inequality, the uniform Gronwall lemma implies
$$
 \|S(t)\theta_0\|^2_{H^{3/2}}\leq \left[2R^2_1+ \frac1\kappa\|f\|^2_{H^1}\right]\e^{\frac{c}{\kappa}R^2_1}, \qquad \forall t\geq 1.
$$
Thus, setting
$$
R_2^2:=\left[2R^2_1+ \frac1\kappa\|f\|^2_{H^1}\right]\e^{\frac{c}{\kappa}R^2_1},
$$
we obtain that 
$$
S(t)B_1\subset B_2, \qquad \forall t\geq 1,
$$
as we wanted.
\end{proof}

We summarize below the consequences of the above result, 
as they follow from \cite{CCP12}*{Proposition 8}.

\begin{corollary}\label{cor:attr}
The dynamical system $S(t)$ generated by \eqref{eq:SQG} on $H^1$ possesses a unique  
global attractor $\AA$ with the
following properties:
\begin{itemize}
	\item $\AA\subset H^{3/2}$ and is the $\omega$-limit set of $B_2$, namely,
	$$
	\AA=\omega(B_2)= \bigcap_{t\geq 0}\overline{\bigcup_{\tau\geq t} S(\tau)B_2}.
	$$
	\item For every bounded set $B\subset H^1$,
	$$
	\lim_{t\to\infty}\dist(S(t)B,\AA)=0,
	$$
	where $\dist$ stands for the usual Hausdorff semi-distance between sets given by the
	$H^1$ norm.
	\item $\AA$ is minimal in the class of $H^1$-closed attracting set.
\end{itemize}
\end{corollary}

\subsection{Invariance of the attractor}\label{sub:invariance}
To conclude the proof of Theorem \ref{thm:globalattra}, we establish the invariance
of the attractor obtained in Corollary \ref{cor:attr}. To this end, we recall the following continuity
result for $S(t)$.

\begin{proposition}[\cite{CTV13}*{Proposition 5.5}]
For every $t>0$, $S(t):B_2\to H^1$ is Lipschitz-continuous in the topology of $H^1$.
\end{proposition}
In other words, the restriction of $S(t)$ to the regular absorbing set $B_2\subset H^{3/2}$
is a continuos map. It turns out that this suffices to complete the proof of Theorem \ref{thm:globalattra}.
\begin{proposition}\label{prop:attr}
The global attractor $\AA$ of $S(t)$ is fully invariant, namely
$$
S(t)\AA=\AA, \qquad \forall t\geq0.
$$
Moreover, $\AA$ is maximal in the class of $H^1$-bounded invariant sets.
\end{proposition} 

\begin{proof}[Proof of Proposition~\ref{prop:attr}]
This proof is classical, so we only sketch here some details. 
Since the global attractor is
the $\omega$-limit set of $B_2$, we have that 
$$
\AA=\omega(B_2)= \big\{\eta\in H^1:  S(t_n)\eta_{n}\to \eta \text{ for some }  \eta_n\in B_2,\ t_n\to \infty   \big\}.
$$
According to \cite{CCP12}*{Proposition 13}, full invariance of $\AA$ follows if one can 
show that $\AA\subset S(t_0)\AA$ for some $t_0>0$. 
Since $B_2$ is absorbing, we may fix $t_0>0$ such that $S(t)B_2\subset B_2$ for all $t\geq t_0$.
Let $\eta\in \omega(B_2)$. Then
there exist $t_n\to\infty$ and $\eta_n\in B_2$ such that
$$
S(t_n)\eta_n\to \eta \qquad \text{as } n\to \infty, \text{ strongly in } H^1.
$$
We may suppose  $t_n\geq 2t_0$ for every $n\in \N$.
Since $\omega(B_2)$ is attracting,
we get in particular
$$
\lim_{n\to\infty}\dist (S(t_n-t_0)B_2,\omega(B_2))=0,
$$
which in turn implies
$$
\lim_{n\to\infty}\left[\inf_{\xi\in \omega(B_2)} \|S(t_n-t_0)\eta_n-\xi\|_{H^1}\right]=0.
$$
So there is a sequence $\xi_n\in \omega(B_2)$ such that 
$$
\lim_{n\to\infty}\left[\|S(t_n-t_0)\eta_n-\xi_n\|_{H^1}\right]=0.
$$
But $\omega(B_2)$ is compact, thus, up to a subsequence, $\xi_n\to\xi\in \omega(B_2)$, which yields at
once 
$$
S(t_n-t_0)\eta_n\to\xi.
$$
Note that $S(t_n-t_0)\eta_n\in B_2$ since $t_n\geq 2t_0$.
Using the continuity of $S(t_0)$ on $B_2$
$$
S(t_0)S(t_n-t_0)\eta_n\to S(t_0)\xi.
$$
On the other hand,
$$
S(t_0)S(t_n-t_0)\eta_n=S(t_n)\eta_n\to\eta.
$$
We conclude that $\eta=S(t_0)\xi$, i.e., $\eta\in S(t_0)\omega(B_2)$. Hence,
$\AA\subset S(t_0)\AA$, and full invariance follows. Once this is established, 
the maximality with respect to invariance is classical. 
\end{proof}

\section*{Acknowledgements}
The work of PC was supported in part by the NSF grants DMS-1209394 and DMS-1265132, 
MCZ was supported in part by an AMS-Simons Travel Award,
while the work of VV was supported in part by the NSF grants DMS-1348193 and DMS-1514771, and an Alfred P. Sloan Research Fellowship.



\end{document}